\documentclass[a4paper,12pt]{article}
\usepackage{graphicx} %
\usepackage{amsxtra, amsbsy,amsgen,amsopn,  amsfonts, amssymb, amsthm}
\usepackage{mathrsfs}
\usepackage{ascmac} 
\usepackage{fancybox}
   \newtheorem{thm}{Theorem}[section]%
    \newtheorem{prop}[thm]{Proposition}%
    \newtheorem{cor}[thm]{Corollary}%
    \newtheorem{lem}[thm]{Lemma}%
    \newtheorem{rem}[thm]{Remark}%

    \numberwithin{equation}{section}
\begin{document}
	\begin{center}
		{\bf 	Another approach to WKB analysis}  \\    
					Sunao {\sc Ouchi} 
		\footnote{ Sophia Univ. Tokyo Japan,  e-mail s\_ouchi@sophia.ac.jp 
    \par \quad Key words and Phrases: WKB analysis, Singular perturbation,  Asymptotic analysis.
	\par \quad 2020 Mathematical Classification Numbers: 34E20, Secondary 34M30, 34M60}
	\end{center}
{\bf Summary.}
  	 A singular perturbation problem called WKB equation \par
	  (Eq) \;\; $ h^2u(x,h)-Q(x)u(x,h)=0$ is studied. $h>0$ is a small parameter. \\ 
		Investigation of (Eq) has long history.
		 Recently it has developed by a new method named "Exact WKB Analysis"
		   based on Borel resummation method and new analytic results.  
	 	Here we study (Eq) by another elementary method. 
		We only apply advanced calculus and the theory of 
		  differential equations to (Eq). 
 		We neither assume turning points are simple nor there is 
		no Stokes curve that connects two turning points. 
\section{\normalsize Introduction}
	  In the present paper we treat  
			\begin{equation}
				\begin{aligned}
					h^2y''-Q(x)y=0
				\end{aligned} \label{WKB'}
			\end{equation}
		$h>0$ is a small parameter.
	 	Equations such as \eqref{WKB'}  appeared in the theory of	the quantum mechanic and 
		is named Schr\"{o}dinger equation,
		It is a singular perturbation problem with respect to a small parameter $h$ (\cite{F}).  
		It is well known that there exist formal solutions of \eqref{WKB'}  called WKB 
	 (Wentzel-Kramers-Brillouin) approximation
		represented by $\exp (  \frac{1}{h}\int^x S(x,h) dx)$, 
		$S(x,h)=\sum_{n=0}^{+\infty} s_n (x)h^n$  is a formal series of $h$.   
		One of the method to study  \eqref{WKB'} has been usage of WKB  approximation.   
		In general the series are divergent.  It seems that they were not so often used
		rigorously at the view point of mathematical analysis. \par
		There was a pioneering work by Voros \cite{V},
		 using the application of Borel resummation 
		of formal series and new analytical theories (\cite{Bal} \cite{Cos} and \cite{E}).   
		After \cite{V}  this new kind of WKB analysis has been used and  have developed (\cite{DDP}. \cite{DP} ). 
		 It called  {\it exact WKB analysis} or {\it complex WKB analysis}.   
		Recently it has been made  more progress,  in particular  connecting with
		microlocal analysis. It has been applicable method for singular  perturbed problems 
		and related fields (e.g.\cite{AIT}, \cite{AT}, \cite{I},  \cite{KaKo}, \cite{KT}  
		and \cite{Ta} (see also references in \cite{KT} and \cite{Ta}).   \par
		In this paper  we study  rigorously and simply \eqref{WKB'}, but do not use 
		Borel resummation  method. 
		 We mainly apply advanced calculus and the theory of ordinary differential equations to 
		\eqref{WKB'}. We show the existence of true solution with asymptotic expansion and 
		give what follows from our investigation. 	
  			The contents are following:
	\begin{enumerate}
		\item[\S 1]  Preliminary.  Elementary notions, tuning point, Stokes curve and Stokes 
		domain are defined.  Formal solution $S(x,h)$ is introduced, which satisfies Riccati equation 
		  $hS'+S^2-Q(x)=0$.  We transform it to  an integral equation.  	
		\item[\S 2]   We introduce an auxiliary ordinary differential equation.  
			It is important and useful in the present paper.  
		\item[\S 3]   Construction of a solution of the integral equation.
		\item[\S 4] We construct a solution of $h^2y''-Q(x)y=0$, by using the solution of
			the integral equation.
		\item	[\S 5] Remarks on Stokes domains.  We show that Stokes domains are conformally 
			equivalent to horizontal strips. 
	\end{enumerate}
	\quad	Appendix.  We give the proof of Proposition \ref{prop2.4}.  	
	\section{\normalsize Preliminary } 
			 We assume $Q(x)$ is a polynomial with degree $m$. 
		 	Let $T_p=\{x; Q(x)=0\}$.  
			A point $x \in T_p$ is called a turning point. 
			Let $ {\mathfrak S}_a(x)=\int_a^x\sqrt{Q(\tau)}d\tau$.
			Let $a$ be a turning point. Then  Stokes curve emanating from $a$ is defined by	
		\begin{equation}
			\begin{aligned}
				\{x; {\Im} {\mathfrak S}_a(x)=0\}.	
			\end{aligned}
		\end{equation}
		There are plural Stokes curves emanating from $a$. 	It ends at $\infty$ or another turning point. Stokes curve is 
		a half closed curve that includes emanating point $a$ and excludes the end point.   
	       If Stokes  curve ${\mathcal C}$ emanating from $a$ ends at another  turning point $b$,  then  
		we joint  ${\mathcal C}$   with   Stokes curves emanating from $b$. 
 (See Figure. 1 (c)) \;      
		$\mathbb{C}$  is divided to several region by Stokes curves.  Each of them is called Stokes domain. 
		We may roughly say that it is surrounded by jointed Stokes curves.  The following figure of Stokes curves 
		is quoted from Iwaki \cite{I}.  \vspace{5mm}\par \qquad
\includegraphics [width=10cm, height=4cm] {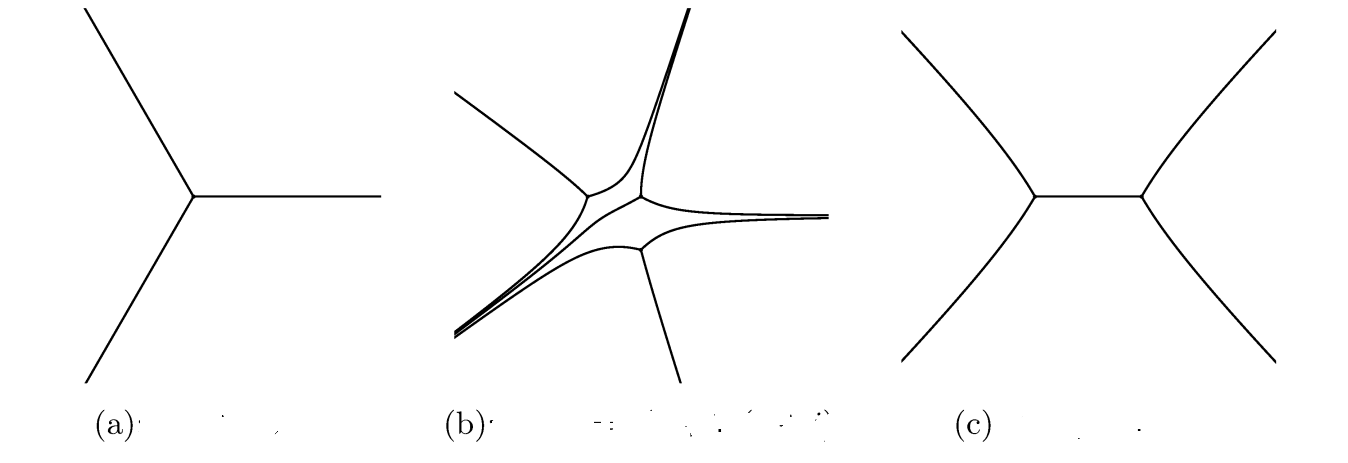} \; Figure.1  \\ 
   (a)\	$Q(x)=x$ \quad (b)\	$Q(x)=x(x+1)(x+i)$ \quad (c)\ $Q(x)=1-x^2$ \vspace{3mm}  \\
	${\mathscr O}(U)$ is the set of holomorphic functions on $U$. 
	 $\partial U$ is the boundary of $U$. 
	\subsection{\normalsize Formal solution}
		We try to find a solution of 
			\begin{equation}
				\begin{aligned}
					h^2y''-Q(x)y=0,
				\end{aligned} \label{WKB}
			\end{equation} in the form $y=\exp \frac{1}{h}\int^x S(\tau)d\tau$. 
		Then we have a Riccati type equation from \eqref{WKB}
		\begin{equation}
			\begin{aligned}
			 hS'+S^2-Q(x)=0.   \label{Ric}
			\end{aligned}
		\end{equation}			
		 There exist 2 formal solutions 
			$\widetilde{S}(x,h)=\sum_{n=0}^{\infty}s_n(x)h^n$ of \eqref{Ric}, 
		\begin{equation*}\left \{
			\begin{aligned}
 				 & s_0(x)=\pm\sqrt{Q(x)}, \;s_1(x)= =-\frac{s_0'(x)}{2s_0(x)}=-\frac{Q'(x)}{4Q(x)}  \\
					 & s_2(x)=-\frac{s_1(x)^2+ s_{1}'(x)}{2s_0(x)}, 
					\; s_1(x)^2+ s_{1}'(x)=\frac{5Q'(x)^2-4Q''(x)Q(x)}{16Q(x)^2} \\
					 & s_n(x)=-\frac{1}{2s_0(x)}\big(\sum_{\begin{subarray}\ i+j=n \\ i,j>0
			\end{subarray}}
					s_i(x)s_j(x)+ s_{n-1}'(x)\big),  \\
			   \end{aligned}  \right.
		\end{equation*}
 				Let  $S(x,h)=s_0(x)+s_1(x)h+W(x,h)h$. Then $W(x,h)$ satisfies
			\begin{equation}
				\begin{aligned}\
				 W'+\big(\frac{2s_0(x)}{h}-\frac{Q'(x)}{2Q(x)}\big)W+ W^2=C(x),
					\ C(x)=-\big(s_1'(x)+s_1(x)^2\big) 
				\end{aligned} \label{W-dif}
			\end{equation}
			$C(x)$ does not depend on the choice of $s_0(x)$.
			We solve \eqref{W-dif} in the following sections,
	\section{\normalsize 	 Integral equation}
		Let $\Omega$ be a Stokes domain.  $a \in \partial \Omega$ is a turning point  and
		${\mathfrak S}_a(x)=\int_{a}^x \sqrt{Q(\tau)}d\tau$. 
			We transform \eqref{W-dif} to an integral equation.  
		Let us introduce integral kernels $K_{\pm}(x,y,h)$, 
			\begin{equation}\left \{
		 		\begin{aligned}	
		       &  K_{+}(x,y,h)=\big(\frac{\sqrt{Q(x)}}{\sqrt{Q(y)}}\big)
		 		  \exp(\frac{ 2({\mathfrak S}_a(y)  -{\mathfrak S}_a(x) )}{h}) 
				 \;\; s_0(x)=\sqrt{Q(x)}, \\
  			&  K_{-}(x,y,h)=\big(\frac{\sqrt{Q(x)}}{\sqrt{Q(y)}}\big)
		 		  	\exp(\frac{ - 2({\mathfrak S}_a(y) -{\mathfrak S}_a(x) )}{h})  
			\;\; s_0(x)=-\sqrt{Q(x)}.
		 		\end{aligned}\right. \label{Ker}
		 	\end{equation}
		We note that if $V_{\pm}(x,h)= \int_{c}^x K_{\pm}(x,y,h)f(y)dy$ is well defined, 
		then  $V_{\pm}(x,h)$  satisfies 
		$V'_{\pm}+\big({\pm}\frac{2\sqrt{Q(x)}}{h}-\frac{Q(x)'}{2Q(x)}\big)V_{\pm}=f(x)$ 
		(see Proposition \ref{prop2.4} and Appendix).
		 Hence  we transform \eqref{W-dif} to an integral equation 
	 		\begin{equation}
			 \begin{aligned}	 		 	
		 		 W_{\pm}(x,h)=-\int_{\gamma_x(\pm)} K_{\pm}(x,y,h)W_{\pm}(y,h)^2dy
			+\int_{\gamma_x(\pm)} K_{\pm}(x,y,h)C(y)dy.
		 		\end{aligned}   \label{w-inteq}
		 	\end{equation}
			$\gamma_x(\pm)$ is an infinite path  from $\infty$ to $x$ in $\Omega$.  It is defined  
			in the next subsection. 
	 \subsection{\normalsize Auxiliary ordinary differential equation} 
			We assume  $	\Im {\mathfrak S}_a(x)>0$ for $x\in \Omega$.  There exist the following cases.
	     \begin{enumerate}
			\item[(1)]   The connected component of $\partial\Omega$ 
			is one curve ${\mathcal C}$ with a turning 
			point  $a \in { \mathcal C}$. ${\Im} {\mathfrak S}_a(x)=0$ for $x\in \partial\Omega$. 
		    	\item[(2)]    The connected component of $\partial\Omega$ consists of two curves 
	    			 $\partial\Omega=\mathcal{C}_1\cup \mathcal{C}_2$,  
			$\mathcal{C}_1\cap \mathcal{C}_2 =\emptyset$.  There are turning points  
			$a \in \mathcal{C}_1 $ and 
				$b \in \mathcal{C}_2$. 
				${\Im} {\mathfrak S}_{a}(x)=0$ for $x\in \mathcal{C}_1$ and
		 		${\Im} {\mathfrak S}_{a}(x)={\Im} {\mathfrak S}_a(b)$ for  $x\in \mathcal{C}_2$..  
			 \end{enumerate}
			There may exist other turning points different from $a, b$ on $\partial \Omega$. 
			 Let  $\delta,\epsilon>0$ and
			\begin{equation} \left \{
				\begin{aligned}
			\mu = & \sup_{x \in \Omega} {\Im}  {\mathfrak S}_a(x),  \; \; 0<\mu \leq +\infty,  \\
			 \Omega_{\delta}=&\{x\in \Omega; \delta <{\Im}  {\mathfrak S}_a(x)<\mu-\delta\}, \\
	   		 D_{\epsilon}=&\{x\in {\mathbb C};   |x-c|>\epsilon \mbox{ for all turning point $c$}\}.
			\end{aligned} \right .
		\end{equation}
		      We have  $\mu =+\infty$ for the case (1) and 
			 $\mu ={\Im}  {\mathfrak S}_a(b)$ for the case (2). 		
			As for constant $\mu$ we refer to Lemmas \ref{lem5.1} and \ref{lem5.2}.
			We remark that there exists $C_{\epsilon}>0$ such that 
		$|Q(x)|^{-1} \leq C_{\epsilon}$ in $D_{\epsilon}$ and for any $\delta>0$ 
		there exists $\epsilon>0$ such that $ \Omega_{\delta}\subset  D_{\epsilon}$. 	\par
			Let us proceed to define integration path $\gamma_x(\pm)$ in \eqref{w-inteq}. 
		 	Consider the following autonomous differential equation,
	 	\begin{equation}
	 		\begin{aligned}\ &
	 			 y'(z)=\frac{1}{\sqrt{Q(y)}}, \;\; y(0)=x  
				 \;\; z\in \mathbb{C}
	 		\end{aligned} \label{path.eq}
	 	\end{equation}
			in order to define  path $\gamma_x(\pm)$. Let $y=\varphi(z,x)$. 	
		 Since equation  \eqref{path.eq} is autonomous,  $\varphi(z_1+z_2,x)=\varphi(z_1,\varphi(z_2,x))$ holds
		and it has an integral ${\mathfrak S}_a(y)-z=C$. Hence  
		\begin{equation}
			\begin{aligned}
		{\mathfrak S}_a(\varphi(z,x))={\mathfrak S}_a(x)+z,
			\end{aligned} \label{2.3}
		\end{equation}	
    	$ \partial_x {\mathfrak S}_a(\varphi(z,x))={\mathfrak S}_a'(x)$  and	
	  \begin{equation}
 	    	\begin{aligned}
  		   	 & \varphi_x(z,x)=\frac{\sqrt{Q(x)}}{\sqrt{Q(\varphi(z,x))}}=\sqrt{Q(x)}\ \varphi_z(z,x).
 	    	\end{aligned} \label{2.4}
  	   \end{equation}
	   Let $z=t+is,  t,s \in {\mathbb R}$. Then
	 	\begin{equation}
	 		\begin{aligned}	 		 
			{\Re} {\mathfrak S}_a(\varphi(z,x))={\Re} {\mathfrak S}_a(x)+ t,\;
			 \Im {\mathfrak S}_a(\varphi(z,x))={\Im} {\mathfrak S}_a(x)+ s.  
	 		\end{aligned}  \label{2.5}
	 	\end{equation}
		We note that if $s=0$, $ {\Im} {\mathfrak S}_a(\varphi(t,x))$ is invariant. 
	\begin{lem}  {\rm (1)}  Let $x\in \Omega$.  Then $\varphi(t,x)$ exists in $-\infty<t<+\infty$. \\
	 {\rm (2)}	Let ${\mathcal C} \subset \partial \Omega$ be a Stokes curve emanating from $a$ and ends at $\infty$.  
	Suppose	$\Re {\mathfrak S}_a(x) >0\  (<0) $ on $\mathcal C-\{a\}$.  
		 Then $\varphi(t,x)$ exists in $ [0,+\infty)$ { $( resp.\ (-\infty, 0] )$} for  $x\in \mathcal C-\{a\}$
		\label{lem2.1}
     \end{lem}
	 \begin{proof} 	 {\rm (1)}. 	Let $x\in \Omega_{\delta}$.   Then $\varphi(t,x) \in \Omega_{\delta}$.
			Assume that $\varphi(t,x)$ exists $[0,T)$ ($T>0$) and
			there is a sequence $\{ t_n\}_{n=1}^{+\infty} \subset [0,T)$ with 
			$\lim_{n\to +\infty} t_n =T$ 
			such that $\varphi(t,x)$ is analytic at $t=t_n$,  singular at $t=T$ 
			and $\lim_{n\to +\infty}\varphi(t_n,x)=c\in  \bar{ \Omega}_{\delta}\cup\{ \infty\}$.  
			Since  $ \varphi_t(t,x)$ is bounded,  $c\not =\infty$. 
			If $c\not =\infty$, then $c\in  \bar{ \Omega}_{\delta}$. For any small $\epsilon>0$ there exists 
			 $n_0$ such that $| \varphi(t_n,x)-c|<\epsilon$ for  $n\geq n_0$.
			Let us consider 
	\begin{equation}
	 		\begin{aligned}\ &
	 			 y'(t)=\frac{1}{\sqrt{Q(y)}}, \;\; y(t_n)= \varphi(t_n,x).
	 		\end{aligned} \label{eq2.6}
	\end{equation}
		   Then the solution of \eqref{eq2.6} exists beyond $t=T$ for a large $n$. Hence
		we can extend  $\varphi(t,x)$ to $t\geq 0$,
		As for the same holds for $t<0$ and $\varphi(t,x)$ exists in $-\infty <t<+\infty$. \\
	 {\rm (2)}. Let  $x \in {\mathcal C}-\{a\}$.  Then $ {\Im} {\mathfrak S}_a(\varphi(t,x))=0 $.  We  may assume	$
		\Re {\mathfrak S}_a(x) >0$.  Then 
		 ${\Re} {\mathfrak S}_a(\varphi(t,x))={\Re} {\mathfrak S}_a(x)+ t \geq {\Re} {\mathfrak S}_a(x)>0$ 
		for  $t\geq 0$.  Hence  $\varphi(t,x) \subset  {\mathcal C}-\{a\}$ for  $t\geq 0$.  By repeating the method
		as (1),  $\varphi(t,x)$ exists $[0,\infty)$.
	\end{proof}
	\begin{rem} Lemma \ref{lem2.1}-{\rm (2)} is used in order that we get Theorem \ref{th4.4}.       	
	\end{rem}
	\begin{lem}  Let $x\in \Omega$.  Then
			\begin{enumerate}
		    \item[{\rm (1)}]  $\lim_{t\to +\infty} \Re {\mathfrak S}_a(\varphi(t,x))=+\infty$.  
		   \item[{\rm (2)}]  $\lim_{t\to +\infty}\varphi(t,x)=\infty$.  
		  \item[{\rm (3)}]  Let $K \subset \Omega$ be compact.  Then
			 $\lim_{t\to +\infty} \min_{x \in K}|\varphi(t,x)|=+\infty$.  
			\end{enumerate} \label{lem2.3}
	\end{lem} 
		\begin{proof}  
		 We note \eqref{2.5}.  
		\begin{enumerate}
		   \item[{\rm (1)}]  This follows from \eqref{2.5}.  
		  \item[{\rm (2)}]  Let $x\in \Omega_{\delta}$.
		 Assume  $\lim_{t\to +\infty}\varphi(t,x) \not=\infty$.
			Then there exist
			$t_1<t_2< \cdots \to +\infty$ and $c\not=\infty$ such that 
			 $\lim_{k\to +\infty}\varphi(t_k,x)=c\in  \overline{\Omega}_{\delta}$.  We have a contradiction 
		$$ {\mathfrak S}_a(c)=\lim_{k \to +\infty} {\mathfrak S}_a(\varphi(t_k,x))
			=\lim_{k \to +\infty}( {\mathfrak S}_a(x)+t_k)=\infty.$$ 
		  \item[{\rm (3)}]  Assume 	 $\lim_{t\to +\infty} \min_{x \in K}|\varphi(t,x)| \not=+\infty$.  
			Then there exist
			$t_1<t_2< \cdots \to +\infty$, a sequence $\{x_k \in K; k=1,2, \cdots\}$ such that
			$\lim_{k\to \infty}x_k=x_* \in K$,  $\varphi(t_k,x_k)=c_k$ with
		$\min_{x \in K}|\varphi(t_k,x)|= |\varphi(t_k,x_k)|=|c_k|$ and 
				 $\lim_{k\to +\infty}c_k=c\not=\infty $ We have a contradiction  from \eqref{2.3}
		$$ {\mathfrak S}_a(c)=\lim_{k \to \infty} {\mathfrak S}_a(\varphi(t_k,x_k)
		)=\lim_{k \to \infty}( {\mathfrak S}_a(x_k)+t_k)={\mathfrak S}_a(x_*)+\infty.$$ 
	\end{enumerate} 
	\end{proof} 
	Similar results hold for $t \to -\infty$.  
 	  Let $\Gamma^* =\{ y=\varphi(t,x); -\infty<t \leq 0 \}$ and $\Gamma_*=\{ y=\varphi(t,x); 0\leq t<+\infty \}$.
	We define $\int_{\gamma_x(\pm)}F(y)dy$.  It is integration from $\infty$ to $x$ on $\gamma_x(+)=\Gamma^*$
	 or $\gamma_x(-)=\Gamma_*$ 	and defined as follows
		\begin{equation}
			\begin{aligned} &
		\int_{\Gamma^*} F(y)dy=: \int_{-\infty}^0 F(\varphi(t,x))\varphi_t(t,x)dt=\int_{-\infty}^0 
		\frac{F(\varphi(t,x))}{\sqrt{Q(\varphi(t,x))}}dt, \\
		&
		\int_{\Gamma_*} F(y)dy=: \int_{+\infty}^0 F(\varphi(t,x))\varphi_t(t,x)dt=\int_{+\infty}^0 
		\frac{F(\varphi(t,x))}{\sqrt{Q(\varphi(t,x))}}dt.
			\end{aligned}
		\end{equation}	
		Let $s_0(x)=\sqrt{Q(x)}$ and $\gamma_{x}(+)={\Gamma^*}$. 	
		Then we have 
	     \begin{equation*}
	     	\begin{aligned}
		 (K_{+}f)(x,h) &=\int_{\Gamma^*} K_{+}(x,y,h)f(y)dy    \\
			=\sqrt{Q(x)} & \int_{\Gamma^*}\exp(\frac{2{\mathfrak S}_a(y)-{2\mathfrak S}_a(x)}{h})\frac{f(y)}{\sqrt{Q(y)}}dy \\
			=\sqrt{Q(x)} & \int_{-\infty}^0\exp(\frac{2{\mathfrak S}_a(\varphi(t,x)-{2\mathfrak S}_a(x))}{h})
			\frac{f(\varphi(t,x))}{Q(\varphi(t,x))}dt
	      	\end{aligned}
	\end{equation*}
 	 It follows from  ${\mathfrak S}_a(\varphi(t,x))={\mathfrak S}_a(x)+t$ that  we have
	 \begin{equation}
	     	\begin{aligned}\ 
			 (K_{+}f)(x,h)  
			={\sqrt{Q(x)}}\int_{-\infty}^{0}\exp(\frac{2t}{h})\frac{f(\varphi(x,t))}{Q(\varphi(t,x))} dt,
	      	\end{aligned} \label{Kf}
	\end{equation}  
		by changing $y=\varphi(t,x)$.  
		 We also have for $s_0(x)=-\sqrt{Q(x)}$, by taking $\gamma_x(-)=\Gamma_*$, 
	     \begin{equation}
	     	\begin{aligned}
		 (K_{-}f)(x,h) &=\int_{\Gamma_*} K_{-}(x,y,h)f(y)dy    \\
			=\sqrt{Q(x)}&  \int_{\Gamma_*}\exp(\frac{-{2\mathfrak S}_a(y)+2{\mathfrak S}_a(x)}{h})\frac{f(y)}{\sqrt{Q(y)}}dy. \\
          		= {\sqrt{Q(x)}} &\int_{\infty}^{0}\exp(-\frac{2t}{h})\frac{f(\varphi(x,t))}{Q(\varphi(t,x))} dt.
	      	\end{aligned}
	\end{equation}
		We study $K_{+}$  in the following. Similar results hold for  $K_{-}$.  
		Let
	\begin{equation*}
			\begin{aligned} \ &
		 \Theta_f (x)=\sup_{ t\leq 0} \big |\frac{f(\varphi(x,t))}{Q(\varphi(t,x))} \big|  \\
		&	g(x,h)=\int_{-\infty}^{0}\exp(\frac{2t}{h})\frac{f(\varphi(x,t))}{Q(\varphi(t,x))}dt 
			\end{aligned} 
		\end{equation*}		 
	Then	 we have 		
     \begin{lem}  
 	     	 $(K_{+}f)(x,h)=\sqrt{Q(x)}g(x,h)$ and $|g(x,h)|\leq \dfrac{h\Theta_f(x)}{2}$.
 	     \end{lem}
	\begin{proof} We have 
	 \begin{equation*}
			\begin{aligned}
			|g(x,h)|\leq \int_{-\infty}^{0}\exp(\frac{2t}{h})\Theta_f(x) dt \leq \frac{h\Theta_f(x)}{2}.
			\end{aligned}
		\end{equation*}			
	\end{proof} 
	The following holds.  The proof is given in Appendix.
	  \begin{prop}  Let  $f(x)\in {\mathscr O}(\Omega)$.  Suppose that for any compact set $L\subset \Omega$ 
		there exists $M_L>0$ such that   
		     \begin{equation}
		  \sup_{x\in L, t\leq 0}\big|\frac{f(\varphi(x,t))}{Q(\varphi(t,x))}\big|\leq M_L. \label{f.est}
		     \end{equation} Then  $(K_{+}f)(x,h)\in {\mathscr O}(\Omega)$ and 
  	   	\begin{equation}
  	   		\begin{aligned}
  	   		\frac{d}{dx}(K_{+}f)(x,h)+\big(\frac{2\sqrt{Q(x)}}{h}-\frac{Q'(x)}{2Q(x)}\big)(K_{+}f)(x,h)=f(x).
     			\end{aligned}
     		\end{equation}  \label{prop2.4}
     \end{prop}
 \section{\normalsize Construction of a solution to integral equation \eqref{w-inteq}}  
   	    Let $\Omega$ be a Stokes domain and we assume $0<\Im { \mathfrak S}_a(x)<\mu$ in $\Omega$ and
		 $s_0(x)=\sqrt{Q(x))}$. 
	   $K_{+}f$ is the integral operator defined by 
 	   \begin{equation}
		 	\begin{aligned}\	 		 	
	 		& (K_{+}f)(x,h)=\sqrt{Q(x)}\int_{-\infty}^{0} \exp(\frac{2t}{h})
			\frac{f(\varphi(t,x))}{Q(\varphi(t,x))}dt.
		 	\end{aligned}
  	   \end{equation}
    	 Let  us return to \eqref{w-inteq} and we proceed to construct a solution of integral equation 
		 \begin{equation}\left \{
			\begin{aligned}\	 		 	
			 		& W(x,h)=-(K_{+}W^2)(x,h)+(K_{+}C)(x,h) \\ & C(x)=-(s_1'(x)+s_1(x)^2).
			\end{aligned} \right. \label{3.2}
		\end{equation} 
	 	We introduce an auxiliary parameter $\varepsilon$ to clarify construction process
	     \begin{equation}
		     	\begin{aligned}
			     	W(x,h)=-(K_{+}W^2)(x,h)+\varepsilon(K_{+}C)(x,h).
		     	\end{aligned}  \label{3.3}
	     \end{equation}
		We will take $\varepsilon=1$ and construct a solution of \eqref{3.2}. 
	      Let $W(x,h,\varepsilon)=\sum_{n=1}^{\infty}W_n(x,h)\varepsilon^n$. Then 
      \begin{equation}
      	 	\begin{aligned}\ W_{1}= K_{+}C,	\;\; 	 
			W_{n}=-\sum_{i+j =n}K_{+}W_iW_j \;\; n\geq 2.	 	
		\end{aligned}  \label{3.4}
	\end{equation}
 		Let   
		     \begin{equation}
		     \Theta(x)=\sup_{ t\leq 0 }\big |\dfrac{C(\varphi(t,x))}{Q(\varphi(t,x))}\big |. \label{Theta}
	     \end{equation}
		It is used  to estimate $W_n(x,h)$. We have
	\begin{lem}
		\begin{enumerate}
				\item[{\rm (1)}]  $\dfrac{C(x)}{Q(x)}$ is a rational function.
				\item[{\rm (2)}]	Let $\epsilon>0$ and $x\in D_{\epsilon}$.  Then there is a constant 
			$C_\epsilon>0$  such that
					\begin{equation}
			\begin{aligned}
			\big	|\frac{C(x)}{Q(x)}\big|\leq \frac{C_\epsilon}{(1+|x|)^{m+2}}.
			\end{aligned}
		      \end{equation}
				\item[{\rm (3)}]  Let $\delta>0$ and $x\in \Omega_{\delta}$.  Then there is a constant 
				$C'_\delta>0$  such that
			\begin{equation}
				\big	|\frac{C(\varphi(t,x)}{Q(\varphi(t,x))}\big|\leq C'_\delta. \label{3.7}
			\end{equation}
		\end{enumerate}	\label{lem3.1}
	\end{lem}
		\begin{proof}  Obviously  $\frac{C(x)}{Q(x)}$ is a rational function.  We have
		  $\frac{C(x)}{Q(x)}=O(x^{-m-2})\;  x\to \infty $, and there is a constant $c_\epsilon>0$ such that 
		$|Q(x)|\geq c_\epsilon $ for $x\in D_{\epsilon}$,  hence (2) holds.  Since  $\varphi(t,x) \subset  \Omega_{\delta}$
		for  $x\in \Omega_{\delta}$,   (3) follows from (2).
		\end{proof}
   	\begin{prop} 
	 		There exist $g_n(x,h)\in \mathscr{O}(\Omega)$ $(n\geq 1)$ and positive constants $M_n$ such that 
			$ W_{n}(x,h)=\sqrt{Q(x)}g_n(x,h)$,
	 	\begin{equation}
	 		\begin{aligned}
			 	|g_n(x,h)|\leq M_n (\frac{h}{2})^{2n-1} \Theta(x)^{n}\quad x\in \Omega
	 		\end{aligned} \label{Wn-est}
	 	\end{equation}
		 	and $\sum_{n=1}^{\infty}M_n{\tau}^n$ converges in a neighborhood of $\tau=0$.
		\label{prop3.2}
 	\end{prop}
	 	\begin{proof}
		 	We have 
	 	 \begin{equation}
	 		\begin{aligned}\	 		 	
	 		&  W_{1}(x,h)
	 		=\int_{-\infty}^{0}\frac{\sqrt{Q(x)}}{Q(\varphi(t,x))}\exp(\frac{2t}{h})C(\varphi(t,x))dt \\
	 		& g_1(x,h)=\int_{-\infty}^{0}\exp(\frac{2t}{h})\frac{C(\varphi(t,x))}{Q(\varphi(t,x))}dt,
	 		\end{aligned} \label{3.8}
   		  \end{equation}
		    $ W_{1}(x,h)=\sqrt{Q(x)}g_1(x,h)$ and  $g_1(x,h)\in \mathscr{O}(\Omega)$.   
			$ |g_1(x,h)|\leq \frac{h}{2} \Theta(x)$ and   $M_1=1$. 
		 	Assume the assertion with \eqref{Wn-est} holds for $1\leq n <N$.  Then for $i+j=N$
		\begin{equation*}
	 		\begin{aligned}
	 		 		K_{+}W_iW_j& =\int_{-\infty}^{0}\frac{\sqrt{Q(x)}}{Q(\varphi(t,x))}\exp(\frac{2t}{h})
 		 		Q(\varphi(t,x)) g_i(\varphi(t,x),h)g_j(\varphi(t,x),h)dt \\
	 		 	 &={\sqrt{Q(x)}}\int_{-\infty}^{0}{\exp(\frac{2t}{h})}g_i(\varphi(t,x),h)g_j(\varphi(t,x),h)dt \\
 			g_N(x.h)&=-({\sum_{i+j=N}K_{+}W_iW_j}) /{\sqrt{Q(x)}}			  
	 		\end{aligned} 
	 	\end{equation*}
		\begin{equation}
	 		\begin{aligned}
	 		 	 &=-\sum_{i+j=N}\int_{-\infty}^{0}{\exp(\frac{2t}{h})}g_i(\varphi(t,x),h)g_j(\varphi(t,x),h)dt.  
	 		\end{aligned} \label{3.10}
	 	\end{equation}
	It follows from $ |g_i(x,h)g_j(x,h)|\leq M_iM_j (\frac{h}{2})^{2N-2} \Theta(x)^{N}$ that
	 		 	\begin{equation*}
	 		\begin{aligned}\ &
	 	  |g_i(\varphi(t,x),h)g_j(\varphi(t,x),h)|\leq M_iM_j (\frac{h}{2})^{2N-2}  \Theta(\varphi(t,x))^{N}.
	 		\end{aligned} 
	 	\end{equation*}
		Since for $t\leq 0$
	 	\begin{equation*}
	 		\begin{aligned}
	 		 \Theta(\varphi(t,x))=\sup_{s\leq 0}\big|\dfrac{C(\varphi(s,\varphi(t,x))}
	 		{Q(\varphi(s,\varphi(t,x)))}\big|=\sup_{s\leq 0}\big|\dfrac{C(\varphi(s+t,x))}
	 		{Q(\varphi(s+t,x))}\big|\leq \Theta(x),
	 		\end{aligned}
	 	\end{equation*}
	 		 	\begin{equation}
	 		\begin{aligned}\ &
	 	  |g_i(\varphi(t,x),h)g_j(\varphi(t,x),h)|\leq M_iM_j (\frac{h}{2})^{2N-2}  \Theta(x)^{N}.
	 		\end{aligned} \label{3.11}
	 	\end{equation}
	 	Hence 
		\begin{equation*}
	 		\begin{aligned}
	 		 |g_N(x,h)|
	 		 	 &\leq\sum_{i+j=N}M_iM_j (\frac{h}{2} )^{2N-2}\Theta(x)^N\int_{-\infty}^{0}{\exp(\frac{2t}{h})}dt \\
				&  \leq\sum_{i+j=N}M_iM_j (\frac{h}{2} )^{2N-1}\Theta(x)^N. 
	 		\end{aligned} 
	 	\end{equation*}
 	We define $M_N=\sum_{i+j=N}M_iM_j$. Then \eqref{Wn-est} holds for $n=N$ and
	 $g_N(x,h) \in {\mathscr O}(\Omega)$ from \eqref{3.7}, \eqref{3.10} and \eqref{3.11}.
      Let us show convergence of $\sum_{n=1}^{\infty}M_n{\tau}^n$.  Let 
      $y=\psi(\tau)$ be a holomorphic function satisfing $y=y^2+\tau M_1$ with $\psi(0)=0$. Let
      $\psi(\tau)=\sum_{n=1}^{\infty} c_n\tau^n$. Then we have
      \begin{equation*}
 	     	\begin{aligned}
      		c_1=M_1, \;\; c_{n}=\sum_{i+j =n}c_ic_j \;\; n\geq 2,
	      	\end{aligned}
      \end{equation*}
	      and $c_n=M_n$. This means that 
	      $\sum_{n=1}^{\infty}M_n{\tau}^n$ converges in a neighborhood of $\tau=0$.
 	\end{proof}
	Take $\varepsilon=1$ and $W(x,h)=\sum_{n=1}^{\infty} W_n(x,h)$.
	 \begin{thm} Let $x \in \Omega$. Then there exist $C,M>0$ such that if  $ 0< h \leq C\Theta(x)^{-1/2}$, 
		  $g(x,h)=\sum_{n=0}^{\infty}g_{n}(x,h)$ and $W(x,h)=\sum_{n=0}^{\infty} W_n(x,h)$ converge,   and
		\begin{equation}
		      	\begin{aligned}
	   		W(x,h)=\sqrt{Q(x)}g(x,h),   \quad |g(x,h)|\leq Mh\Theta(x).
		      	\end{aligned} 
		 \end{equation}\label{th3.3}
	  \end{thm}
	  \begin{proof} 
	 	 Since $\sum_{n=1}^{\infty}M_n{\tau}^n$ converges, there exists $\rho>0$ such that  $M_n\leq \rho^n\; (n\geq 1)$. Then
	 	 \begin{equation}
	 	 	\begin{aligned}
			 |g(x,h)|\leq  \sum_{n=1}^{\infty}|g_n(x,h)|\leq \frac{2}{h}\sum_{n=1}^{\infty}(\frac {h^2\rho}{4})^n\Theta(x)^n
	 	 	\end{aligned}
	 	 \end{equation}
	 	 If $|(\frac {h^2\rho}{4})\Theta(x)|<1/2$, $\sum_{n=1}^{\infty}|g_n(x,h)|$ converges.
	  \end{proof} 
		We have from Lemma \ref{lem3.1}
	 \begin{cor} Let $\delta>0$ and $x \in \Omega_{\delta}$. Then there exist $H(\delta)$,   $M(\delta)>0$ such that 
		if  $ 0< h <H(\delta)$, 
		  $g(x,h)$ 
		is   holomorphic in $\Omega_{\delta}$ and  $|g(x,h)|\leq M(\delta)h.$
	  \end{cor}
		\begin{proof}
		We have $\sup_{x\in \Omega_{\delta} } \Theta(x)\leq  C_\delta$ (Lemma \ref{lem3.1}).  Then
		\begin{equation*}
			\begin{aligned}
		(\frac {h^2\rho}{4})\Theta(x) \leq (\frac {h^2\rho}{4})C_\delta 
			\end{aligned}
		\end{equation*}			
      		If $	|(\frac {h^2\rho}{4}) C_\delta \leq 1/2$,  the assertion hold and  $H(\delta)= (\frac {2}{\rho C_\delta})^{1/2}$.
		\end{proof}    
	Let us study asymptotic behavior of  $g(x,h)$.  
		Let $\psi(x,h)$ be a holomorphic
		function in $x$ with a parameter $h>0$ and assume the following holds: \par
            { \it For any $\delta>0$ there exists $H(\delta)>0$ such that $\psi(x,h)$ is holomorphic in
		 $x\in \Omega_{\delta}$ for $0<h<H(\delta)$.  \par
		$\psi(x,h) \sim \sum_{n=0}^\infty a_n(x)h^n$,  where $\{a_n(x);n\geq 0\}$ are holomorphic in $\Omega$,
 		means that for any $N\geq 0$ there exists $C_N(\delta)$ such that
		\begin{equation*}
			\begin{aligned}
		      |	\psi(x,h) - \sum_{n=0}^{N} a_n(x)h^n| \leq C_N(\delta) h^{N+1}\quad x\in \Omega_{\delta}.
			\end{aligned}
		\end{equation*}		}	
		If $ a_n(x)=0$ for $0\leq n\leq N$,  then  we denote simply $\psi(x,h)=O(h^{N+1})$.  Hence 
		\begin{equation}
			\begin{aligned}
			\psi(x,h) = \sum_{n=0}^{N} a_n(x)h^n +\psi_N(x,h) , \;  \psi_N(x,h)=O(h^{N+1} ).
			\end{aligned}
		\end{equation}	
		Let 
	\begin{equation}
	 		\begin{aligned} 
			&\widehat{g}_1(t,x,h)=\frac{C(\varphi(t,x))}{Q(\varphi(t,x))},  \\
	 	      &  \widehat{g}_{n}(t, x,h)=-\sum_{i+j=n}g_i(\varphi(t,x),h)g_j(\varphi(t,x),h)\;\; n\geq 2.
	 		\end{aligned}
	 	\end{equation}
        Then
	\begin{equation}
	 		\begin{aligned} 
		g_n(x,h)=\int_{-\infty}^0 \exp(\frac{2t}{h}) \widehat{g}_{n}(t, x,h) dt.
	 		\end{aligned}
	 	\end{equation}
		It follows from Watson's Lemma that
		\begin{equation*}
			\begin{aligned}
			g_1(x,h)\sim \sum_{k=1}^{\infty} a_{1,k}(x)h^k,\quad a_{1,1}=\frac{C(x)}{2Q(x)}.
			\end{aligned}
		\end{equation*}	
		We assume $g_n(x,h)\sim \sum_{k=2n-1}^{\infty} a_{n,k}(x)h^k$ for $1\leq n \leq N-1$. 
		Let    $a_{i,j,m}(x)=\sum_{k+\ell=m} a_{i,k}(x)a_{j,\ell}(x)$ for $i+j=N$. Then 
			\begin{equation*}
			\begin{aligned} &
		 (\sum_{k=2i-1}^{\infty} a_{i,k}(x)h^k)  (\sum_{\ell=2j-1}^{\infty} a_{j,\ell}(x)h^\ell) =
			\sum_{m=2N-2}^{\infty}  a_{i,j,m}(x) h^m
			\end{aligned}
		\end{equation*}	
		as formal series.  Hence 
		\begin{equation*}
			\begin{aligned} &
		   g_{i}(x,h) g_{j}(x,h)=	\sum_{m=2N-2}^{M}  a_{i,j,m}(x) h^m+ a_{i,j}^M(x,h)\\ &   a_{i,j}^M(x,h)=O(h^{M+1})
			\end{aligned}
		\end{equation*}			
		and 
	\begin{equation*}
	 		\begin{aligned} 
		&	g_N(x,h)=	
		 \sum_{i+j=N}   \int_{-\infty}^{0} {\exp(\frac{2t}{h})}g_{i}(\varphi(t,x),h) g_{j}(\varphi(t,x),h)dt \\
		= &	\sum_{m=2N-2}^{M}\Big ( \sum_{i+j=N}\int_{-\infty}^{0}  {\exp(\frac{2t}{h})}a_{i,j,m}(\varphi(t,x))dt \Big ) h^m \\
              &  +  \sum_{i+j=N}\int_{-\infty}^{0}  {\exp(\frac{2t}{h})} a_{i,j}^M(\varphi(t,x),h)dt
			\end{aligned}
	 	\end{equation*}	
		By Watson's lemma there  exist $\{a_{N.k}(x); k\geq 2N-1\}$ such that  
		$g_N(x,h)\sim 	\sum_{k=2N-1}^{\infty} a_{N,k}(x)h^k$.  \par
		Consequently the following holds for $g(x,h)$ in Theorem \ref{th3.3} 
	\begin{thm} $g(x.h)$ has an asymptotic expansion $g(x.h)\sim  \sum_{k=1}^{\infty} a_{k}(x)h^k.$
	\end{thm}
	\begin{proof}  We have $g(x.h)=\sum_{n=1}^{\infty}g_n(x,h)$ and a relation of formal series
		\begin{equation*}
			\begin{aligned}
			\sum_{n=1}^{\infty}(\sum_{k=2n-1}^{\infty} a_{n,k}(x)h^k)
				=\sum_{k=1}^{\infty}\big(\sum_{n=1}^{[(k+1)/2]} a_{n,k}(x)\big)h^k=\sum_{k=1}^{\infty} a_{k}(x)h^k.
			\end{aligned}
		\end{equation*}			
	\end{proof}
	As for $S(x,h)=s_0(x)-\frac{Q'(x)}{4Q(x)}h +hW(x,h), \; W(x,h)=\sqrt{Q(x)}g(x,h)$
	\begin{cor}
		\begin{equation}
			\begin{aligned}
		      {S(x,h)} \sim s_0(x) -\frac{Q'(x)}{4Q(x)}h+\sum_{n=2}^{\infty}s_n(x)h^{n}.
			\end{aligned}
		\end{equation}			
	\end{cor}
	\begin{proof}
		We have $W(x,h) =\sqrt{Q(x)}g(x,h) \sim \sqrt{Q(x)}( \sum_{k=1}^{\infty} a_{k}(x)h^k)$ and 
	\begin{equation*}
			\begin{aligned} &
		 {S(x,h)}
		 \sim  s_0(x) -\frac{Q'(x)}{4Q(x)}h+\sqrt{Q(x)}\sum_{k=1}^{\infty} a_{k}(x)h^{k+1}.
			\end{aligned}
		\end{equation*}	
		It follows from uniqueness of formal solution 	that $s_n(x)=\sqrt{Q(x)} a_{n-1}(x)$ for $n\geq 2$.		
	\end{proof}
	\section{\normalsize Solutions of $h^2 y''+Q(x)y=0$}	 
		\subsection{\normalsize Behavior of $\varphi(t,x)$} 
			The purpose of this section is  to define integral $\exp\big(\frac{1}{h} \int^x S_+(\tau,h) d\tau$ and construct
   		a solution of  $h^2 y''+Q(x)y=0$.  As before  $\Omega$ is a Stokes domain and
		 $a \in \partial \Omega$ is a turning point.  
		It is defined by
	 \begin{equation}
	 	 	\begin{aligned}  \  &
			\exp\big(\frac{1}{h} \int^x S_+(\tau,h) d\tau\big)
		      = \frac{1}{Q(x)^{1/4}}	\exp\big(\frac{1}{h}\int_a^x\sqrt{Q(\tau)}d\tau+\int_{\Gamma^*} W_+(\tau,h)d\tau\big) \\
			  &= \frac{1}{Q(x)^{1/4}} 
					\exp \big(\frac{1}{h}\int_a^x\sqrt{Q(\tau)}d\tau+\int^0_{-\infty}g_{+}(\varphi(t,x),h)dt.
	 	 	\end{aligned} \label{4.1}
	 	 \end{equation}
		 $\int_{\Gamma^*} \cdot\cdot\ d\tau$ is integration from $\infty$ to $x$ 	along path ${\Gamma^*}$. 
       	Let us get the behavior of $\varphi(t,x)$ more precisely to show 
		$\int^0_{-\infty}g_+(\varphi(t,x),h)dt$ is integrable.
 	\begin{prop}  Let $K\subset \Omega$ be compact. Then there exist positive constants $C_K, T_K$ 
			     such that $|\varphi(t,x))| \geq C_K|t|^{2/(m+2)}$ for $x\in K$ and $|t|\geq T_K$. \label{prop-decay}
	  \end{prop}
	  \begin{proof}  We may assume $t>0$.  Let $H(y)=\frac{(m/2+1)y^{m/2}}{\sqrt{Q(y)}}$. 	 There exist $c\not=0$
		and $R>0$ such that  $H(y)=\pm c +H_1^{\pm}(y)$ with $ |H_1^{\pm}(y)|\leq |c|/2 $ for $|y|\geq R$.
		 We have  
  	     \begin{equation*}
	       	\begin{aligned}  &
	       	(m/2+1)\varphi(t,x)^{m/2}\varphi'(t,x)=\frac{(m/2+1)\varphi(t,x)^{m/2}}{\sqrt{Q(\varphi(t,x)}} ,\\
	         & \frac{d}{dt}\varphi(t,x)^{m/2+1} 
	         =H(\varphi(t,x)) .  
	       	\end{aligned}
 	      \end{equation*}
			 It follows from Lemma \ref{lem2.3} 
			that for any $R>0$ there is $T=T(K,R)>0$ such that 
			$ | \varphi(t,x)|\geq R$ for $ x\in K$ and $ |t|\geq T$. 
			We may assume that 	 $H(\varphi(t,x))=c +H_1(\varphi(t,x))$ with $ |H_1(\varphi(t,x))|\leq |c|/2 $
			for $|t|\geq T$.   We have  
		\begin{equation}
	 	      	\begin{aligned} \ &
 		      	   \varphi(t,x)^{m/2+1}=\varphi(T,x)^{m/2+1}+c(t-T)+\int_{T}^t H_1(\varphi(\tau,x))d\tau.
 		      	\end{aligned} \label{varT''}
 	      \end{equation}
		Let $t\geq T$.  We have from \eqref{varT''}
		\begin{equation*}
       	\begin{aligned} 
  		& 	|\varphi(t,x)|^{m/2+1} \geq |c|(t-T)/2-|\varphi(T,x)|^{m/2+1}
       	\end{aligned}
       \end{equation*}
		 for  $t >T$. 	Hence there exist $C'>0$ and $T'>0$ such that $|\varphi(t,x)|^{m/2+1}\geq C' t$ for $t>T'$.
  	      It also holds that $|\varphi(t,x)^{m/2+1}| \geq C'| t|$ for $t<-T'' $.  
      \end{proof}
	\begin{cor}  Let $K$ is a compact set in $\Omega$.  Then there exist a constant $M_K>0$ such that
		for $x\in K$ 
		\begin{equation}
		 \Big	|\frac{C(\varphi(t,x))}{Q(\varphi(t,x)} \Big |\leq\frac{M_K}{(1+|t|)^2}, \;   -\infty<t<+\infty. 
		\end{equation} \label{cor-decay'}
	  \end{cor}
		\begin{proof}
		Let $ K \subset \Omega_{\delta}$ and $y \in \Omega_{\delta}$.  Then 
		there exists $C_{\delta}>0$ such that $|\frac{C(y)}{Q(y)}\leq C_{\delta}|/(1+|y|)^{m+2}$ and 
	     $|\frac{C(\varphi(t,x)}{Q((\varphi(t,x))}\leq C_{\delta}/(1+|(\varphi(t,x)^{m+2}|)$.  
	   	 If $|t|\geq T_K$,  $|\varphi(t,x)^{m+2}|\geq C'|t|^2$.  If $|t|\leq T_K$, 
		 $|\frac{C(\varphi(t,x)}{Q((\varphi(t,x))}|\leq C_{\delta}$.  Hence the estimate holds.  
		\end{proof}
	\subsection{\normalsize Construction of solutions }  
	  We construct solutions of \eqref{WKB}.  Let  $\Omega$ be a Stokes domain and $K\subset \Omega$ be a compact	 set.
	 We have from  Corollary \ref{cor-decay'} for $t\leq 0$
	\begin{equation*}
			\begin{aligned} 
		\Theta(\varphi(t,x)) & =\sup_{s\leq 0} 
		\Big	|\frac{C(\varphi(s,\varphi(t,x))}{Q(\varphi(s,\varphi(t,x))}\Big| =\sup_{s\leq 0}
		\Big	|\frac{C(\varphi(s+t,x))}{Q(\varphi(s+t,x))}\Big| \\ &  \leq \sup_{s\leq 0} \frac{M_K}{(1+|s+t|)^2 }.
		 \leq \frac{M_K}{(1+|t|)^2 }.
			\end{aligned}
		\end{equation*}	
		Hence it follows from $|g_+(\varphi(t,x),h)|\leq Mh\Theta(\varphi(t,x))\ (t\leq 0$) 
		that
		$|g_+(\varphi(t,x),h)|\leq M'_K h/(1+|t|)^2$.  
		We  can  define integration $\exp(h^{-1} \int^x S_+(\tau,h) d \tau)$
		 for $0<h<C\Theta(x)^{-1/2}$ as follows:  Take integral  path  $\gamma_x(+)= \Gamma^*$ and
		 \begin{equation}
	 	 	\begin{aligned} 
		U_+(x,h)= &\exp\big(\frac{1}{h} \int^x S_+(\tau,h) d\tau\big) \\
		\overset{\rm def.}{=}\frac{1}{Q(x)^{1/4}} &
			\exp\big(\frac{1}{h}\int_a^x\sqrt{Q(\tau)}d\tau+\int_{\Gamma^*} W_+(\tau,h)d\tau\big) \\
			=\frac{1}{Q(x)^{1/4}} &
					\exp \big(\frac{1}{h}\int_a^x\sqrt{Q(\tau)}d\tau+\int^0_{-\infty} g_+(\varphi(t,x),h)dt\big).
	 	 	\end{aligned}  \label{4.4}
	 	 \end{equation}
	    \quad	There exists another solution of Ricatti equation \eqref{Ric} by taking $s_0(x)=-\sqrt{Q(x)}$.
		We get $ W_{-}(x,h)$, 	by solving
			\begin{equation}
				 W'_{-}-\big(\frac{2\sqrt{Q(x)}}{h}+\frac{Q'(x)}{2Q(x)}\big)W_{-}+ W_{-}^2=C(x)
			 \label{W-}
			\end{equation}
  		in the same way as $W_{+}(x,h)$.  Let
			\begin{equation} \left \{
				\begin{aligned} \ &
				S_{-}(x,h)=-\sqrt{Q(x)}- \frac{Q'(x)}{4Q(x)}h+W_{-}(x,h)h, \\
				&   W_{-}(x.h)=\sqrt{Q(x)}g_{-}(x,h),
				\end{aligned}  \right .
			\end{equation}
  	 	$\gamma_{-}(x)=\Gamma_*$ 
		and 
	 \begin{equation}
	 	 	\begin{aligned} 
		U_-(x,h)= &\exp\big(\frac{1}{h} \int^x S_-(\tau,h) d\tau\big) \\
			= \frac{1}{Q(x)^{1/4}} &	\exp\big(-\frac{1}{h}\int_a^x\sqrt{Q(\tau)}d\tau+\int_{\Gamma_*} W_-(\tau,h)d\tau\big) \\
			=\frac{1}{Q(x)^{1/4}} &
					\exp \big(-\frac{1}{h}\int_a^x\sqrt{Q(\tau)}d\tau+\int^0_{+\infty} g_-(\varphi(t,x),h)dt\big).
	 	 	\end{aligned}  \label{4.7}
	 	 \end{equation}
 		     $\int_{\Gamma_*}\cdot \cdot d\tau$ is integration from $\infty$ to $x$ 
			along path ${\Gamma_*}$.
			 Consequently we have 2 independent  solutions of \eqref{WKB}.
		\begin{thm}  Let $x\in \Omega_{\delta}$. There exist 2 independent  solutions $U_{+}(x.h)$ and $U_{-}(x.h)$ 
			of \eqref{WKB}	for $0<h<H(\delta)$.
		\end{thm}
			Thus we get WKB-genuine solutions $U_{\pm}(x.h)$,  
	\subsection{\normalsize Adjoining Stokes domain}  
			Let  $\Omega$ and $\Omega'$ be  Stokes domains  adjoining each other and $a$ be a mutual turning point
			and ${\mathfrak S}_a(x)= \int_a^x \sqrt{Q(\tau)}d\tau$.
			Suppose that ${\mathcal C}=\partial\Omega \cap \partial \Omega'$ is a Stokes curve emanating from 
			$a$ and ending at $\infty$.  \\ \quad
		  Then $\Im \int_a^x \sqrt{Q(\tau)}d\tau =0$ on $ {\mathcal C}$.  Let
			$ \widetilde{\Omega}= \Omega\cup\Omega'\cup( {\mathcal C}-\{a\})$.  	 We assume
		\begin{equation} 
			\begin{aligned}
  			 \Im \int_a^x \sqrt{Q(\tau)}d\tau >0 \;\; x \in \Omega,  \quad
				\Im \int_a^x \sqrt{Q(\tau)}d\tau <0 \;\; x \in \Omega'  
			\end{aligned} 		\end{equation}	
		Then it follows from Lemma  \ref{lem2.1}-(2) that 
		\begin{enumerate}
			\item[{\rm (1)}] Suppose 
				$ \Re \int_a^x \sqrt{Q(\tau)}d\tau <0 $ for  $ x \in  {\mathcal C}-\{a\}$. Then
			       $\varphi(t,x)$ ($ x \in  \widetilde{\Omega}$) exists on $-\infty<t \leq 0$.
			\item[{\rm (2)}] 
				Suppose 
				$ \Re \int_a^x \sqrt{Q(\tau)}d\tau >0$ for  $ x \in  {\mathcal C}-\{a\}$. Then
			       $\varphi(t,x)$ ($ x \in  \widetilde{\Omega}$) exists  on  $0 \leq t <+\infty $.  
			\end{enumerate}
		Let  $\mu=\sup_{x\in \Omega}\Im {\mathfrak S}_a(x)$ and  
		$\mu'=\inf _{x\in \Omega' } \Im{\mathfrak S}_a(x)$,  $-\infty\leq \mu'<0<\mu \leq +\infty$. 
		 For any small $\delta, \epsilon>0$
	\begin{equation*} 
			\begin{aligned}\ &
			\widetilde{\Omega}_{\delta}=\{x \in \widetilde{\Omega}; 
			\mu'+\delta<\Im{\mathfrak S}_a(x)<\mu-\delta\}.
			\end{aligned} 
	\end{equation*}	 
	If case (1)	,
	\begin{equation*} 
			\begin{aligned}\
			& 	\widetilde{\Omega}_{\epsilon,\delta}(+)=\widetilde{\Omega}_{\delta}-
			 \{x \in \widetilde{\Omega}; \Re {\mathfrak S}_a(x)>-\epsilon, 	|\Im{\mathfrak S}_a(x)|<\delta\}.
			\end{aligned} 
	\end{equation*}
	If case (2),		
	\begin{equation*} 
			\begin{aligned}\
			& \widetilde{\Omega}_{\epsilon,\delta}(-)=\widetilde{\Omega}_{\delta}-
			 \{x \in \widetilde{\Omega}; \Re {\mathfrak S}_a(x)<\epsilon, 
			|\Im{\mathfrak S}_a(x)|<\delta\}. 
			\end{aligned} 
	\end{equation*}		
		Estimate $|\frac{C(x}{Q(x)}|\leq C_{\epsilon, \delta}/(1+|x|)^{m+2}$ holds for 
		$x \in  \widetilde{\Omega}_{\epsilon, \delta}(\pm)$.  \\
\begin{minipage}{7cm}
	\includegraphics [width=70mm, height=50mm] {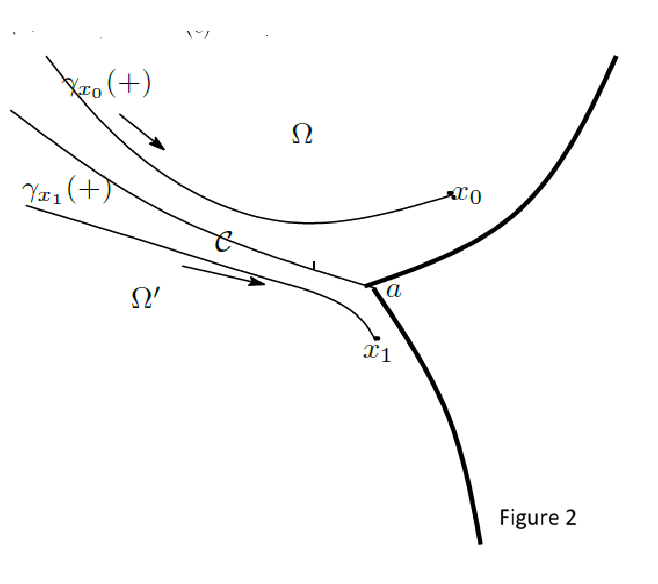}  
\end{minipage}
\hspace{5mm}
\begin{minipage}{5cm}
  	$ \Re \int_a^x \sqrt{Q(\tau)}d\tau <0 $ on  $ {\mathcal C}-\{a\}$ is assumed in Figure 2. 
  Let $x\in {\mathcal C}-\{a\}$. Then $\{\varphi(t,x); t\leq 0\}\subset 
{\mathcal C}-\{a\}$. Hence every orbit $\{\varphi(t,x); t\leq 0\}$ for $x\in \widetilde{\Omega}$ is
included in $\widetilde{\Omega}$.
\end{minipage} \\
	        The existence  and behavior of $\{\varphi(t,x); -\infty< t \leq 0\}$ ( $\{\varphi(t,x); 0\leq t<+\infty \}$)
		are essentially used to  construct $U_{+}(x,h)$ \eqref{4.4} (resp.  $U_{-}(x,h)$ \eqref{4.7}). The preceding method is
		available for $	\widetilde{\Omega}_{\epsilon,\delta}(+)$ (resp. $ \widetilde{\Omega}_{\epsilon,\delta}(-))$.
	\begin{thm}  For any $\epsilon, \delta>0$ there exists $H_{\epsilon,\delta}>0$  such that 
		the  followings hold	for  $0<h<H_{\epsilon, \delta}$.
				\begin{enumerate}
			\item[{\rm (1)}] Suppose that 
				$ \Re \int_a^x \sqrt{Q(\tau)}d\tau <0 $ for $ x \in  {\mathcal C}-\{a\} $. Then  
				$U_{+}(x,h)$ defined by \eqref{4.4} is holomorphically extensible
				to $\widetilde{\Omega}_{\epsilon, \delta}(+)$.
			\item[{\rm (2)}] Suppose that 
				$ \Re \int_a^x \sqrt{Q(\tau)}d\tau >0$ for $ x \in  {\mathcal C}-\{a\}$. Then 
					$U_{-}(x,h)$defined by \eqref{4.7} is holomorphically extensible
				to $\widetilde{\Omega}_{\epsilon, \delta}(-)$. 
			\end{enumerate} \label{th4.4}
		\end{thm}
		\begin{rem}
			Suppose $a$ is simple turning point. Then the same result is given in \cite{V}.  Moreover
		       the connection formula of the other solution $U_{-}(x,h)$ 
			($U_{+}(x,h)$) is given in \cite{V}  under the condition (1)  (resp. the condition (2)) . 
		\end{rem}
	\section{\normalsize Remarks on Stokes domains}  
		We study Stokes domain  in detail.  Let $\Omega$  be  a Stokes domain with turning point $a$ and 
		${\Im } {\mathfrak S}_a(x)>0$  in $\Omega$.
		There are 2 cases (1) and (2).  
	     \begin{enumerate}
			\item[(1)]   
			 $\partial\Omega$ is connected and one curve ${\mathcal C}$ with a turning point  
			$a \in { \mathcal C}$.  ${\Im} {\mathfrak S}_a(x)=0$ for $x\in \partial\Omega$ 
		    	\item[(2)]    The connected component of $\partial\Omega$ consists of two curves 
			$\mathcal{C}_1$ and $ \mathcal{C}_2$,
	    			 $\partial\Omega=\mathcal{C}_1\cup \mathcal{C}_2$,  
			$\mathcal{C}_1\cap \mathcal{C}_2 =\emptyset$.  There are turning points  
			$a \in \mathcal{C}_1 $ and 
				$b \in \mathcal{C}_2$. 
				${\Im} {\mathfrak S}_{a}(x)=0$ for $x\in \mathcal{C}_1$ and
		 		${\Im} {\mathfrak S}_{a}(x)={\Im} {\mathfrak S}_a(b)$ for  $x\in \mathcal{C}_2$ 
			 \end{enumerate}
		Let us use  again 
		\begin{equation}
			\begin{aligned} &
			\frac{dy}{dz}=\frac{1}{\sqrt{Q(y)}}  \quad y(0)=x, \; {Q(x)\not =0} \quad y=\varphi(z,x).     
	              \end{aligned} \label{phi}
		\end{equation}  
			and the following equality 
		\begin{equation}
			\begin{aligned}
 		       {\mathfrak S}_a(\varphi(z,x))= {\mathfrak S}_a(x)+z, \; \;  {\mathfrak S}_a(x)=\int_a^x \sqrt{Q(\tau)}d\tau,		
			\end{aligned} \label{5.2}
		\end{equation}
		where $z \in {\mathbb C}$. 
		We give some fundamental property of 
		$\mu=\sup_{x \in \Omega} \Im  {\mathfrak S}_a(x)$ (see section 2).
		\begin{lem} Let $\Omega$ be type (1).  
	     Then 	$\{{\Im} {\mathfrak S}_a(x); x\in \Omega\} = (0,+\infty)$ and  $\mu=+\infty$.	\label{lem5.1}	
        \end{lem}
	\begin{proof}
     		 Let  $x\in \Omega$ and $z= is,  s \geq 0$ in \eqref{phi}. Then
		\begin{equation*}
			\begin{aligned} &
          		{\Re}  {\mathfrak S}_a(\varphi(is,x))= {\Re}  {\mathfrak S}_a(x), \,\, 
			{\Im} {\mathfrak S}_a(\varphi(is,x))={\Im}  {\mathfrak S}_a(x)+s> 0,
			\end{aligned}   
		\end{equation*}	
		${\Im} {\mathfrak S}_a(x)=0$ for $x \in \partial\Omega$.   Hence 
		${\mathfrak S}_a(\varphi(is,x))\subset \Omega$ and  	
		$\{{\Im}{\mathfrak S}_a(x); x\in \Omega\} =(0,+\infty)$. 
		\end{proof}
		We have in the same way	
   	   \begin{lem} 
		    Let $\Omega$ be a type (2). 
			Then $\{{\Im} {\mathfrak S}_a(x); x\in \Omega\} = (0,{\Im} {\mathfrak S}_a(b))$ and
			$\mu={\Im} {\mathfrak S}_a(b) $. 
		 \label{lem5.2}
	    \end{lem}
	 \begin{thm}    Let $\Omega$ be type (1).  
 		Then $z={\mathfrak S}_a(x)$ is a conformal bijective mapping from $\Omega$ to 
		$\mathbb{C}_{+}=\{z \in \mathbb{C}; {\rm Im}z>0\}$.   	   \label{Conf}
	   \end{thm}
\begin{proof}
		First we show  ${\mathfrak S}_a(x)$ is  surjective.     Let $z=u+iv \in \mathbb{C}_+$. $ u,v \in \mathbb{R}, v>0$.
	We have $\{{\Im} {\mathfrak S}_a(x); x\in \Omega\} =(0,+\infty)$.
         Hence there is $x_0 \in \Omega$ such that
	   ${\Im} {\mathfrak S}_a(x_0)=v$. Let  $\varphi(z, x_0)$ be that defined by \eqref{phi}. 
		It  satisfies \eqref{5.2}.
	      Take $z=u-{\Re} {\mathfrak S}_a(x_0)\in {\mathbb R}$. 
	Then
	   \begin{equation*}
	 		\begin{aligned}	 \ &			
			 {\Re}{\mathfrak S}_a\big(\varphi(u-{\Re}{\mathfrak S}_a(x_0),x_0)\big)=
		\Re {\mathfrak S}_a(x_0)+u-{\Re} {\mathfrak S}_a(x_0)=u,\\
			  &{\Im}{\mathfrak S}_a\big(\varphi(u-{\Re} {\mathfrak S}_a(x_0),x_0)\big)
				=\Im {\mathfrak S}_a(x_0)=v
	 		\end{aligned}
	 	\end{equation*}
		   and ${\mathfrak S}_a\big(\varphi(u-{\Re}{\mathfrak S}_a(x_0),x_0)\big)=u+iv=z$. 
		 ${\mathfrak S}_a(x)$ is surjective.   \par
		     Let us proceed to show  $ {\mathfrak S}_a(x)$ is injective.
			 Let  $\delta>0$ 
			and  $\mathbb{C}_{\delta}= \{z;{\Im} z>\delta\}$. 		
			Let  $x_*\in \Omega_{\delta}$ and ${\mathfrak S}_a(x_*)=z_*$. Then $\Im z_* >{\delta}$.   
			Consider 
  	   \begin{equation}
	  	\begin{aligned}
	   			 \frac{d\Psi(z)}{dz}=\frac{1}{\sqrt{Q(\Psi(z))}}\quad \Psi(z_*)=x_*.
	   	\end{aligned} \label{eq-psi}
	 \end{equation}
	   	$\Psi(z)$ exists in a neighborhood of $z=z_*$ and     	
   	\begin{equation*}
	   		\begin{aligned}
	   			\int_{x_*}^{\Psi(z)} \sqrt{Q(\tau)}d\tau =z-z_*, 
			\; {\mathfrak S}_a({\Psi(z)})-{\mathfrak S}_a({\Psi(z_*)})=z-z_*.
	   		\end{aligned}
	   	\end{equation*}
		Hence ${\mathfrak S}_a({\Psi(z)})=z$. Let $z=u+iv \in \mathbb{C}_{\delta}$, $u,v\in 
		\mathbb{R}$. Then $\Re {\mathfrak S}_a({\Psi(z)})=u$, 
		$\Im {\mathfrak S}_a({\Psi(z)})=v$ and  
		${\mathfrak S}_a({\Psi(z)}) \subset \Omega_{\delta}$ for 
		$z=u+iv \in \mathbb{C}_{\delta}$.  Suppose $\Psi(z)$ is singular at  $z=\hat{z}\in \mathbb{C}_{\delta}$.  Then
		there exists a sequence $\{z_n\in \mathbb{C}_{\delta}; n=1,2,\cdots \}$ and $c$ such that $\lim_{\to \infty}z_n=\hat{z}$
		and  $\lim_{n \to \infty}\Psi(z_n)=c$.  Since $\Psi'(z_n)$ is bounded, $c\not=\infty$. It means that, by solving
		$ \frac{d\Psi(z)}{dz}=\frac{1}{\sqrt{Q(\Psi(z))}}\; \Psi(\hat{z})=c$, $\Psi(z)$ can be extended in a neighborhood of 
		$z=\hat{z}$ and is not singular at  $z=\hat{z}$. Consequently $\Psi(z)$ exists in 
		 $ \mathbb{C}_{\delta}$ for any $ \delta>0$, hence in $\mathbb{C}_{+}$.
	 We have
	   	\begin{equation}
	   		\begin{aligned}
	   		 \frac{d}{dx}\Psi\big({\mathfrak S}_a(x)\big)=\Psi'\big({\mathfrak S}_a(x)\big)
				{\mathfrak S}'_a(x)= \frac{\sqrt{Q(x)}}{\sqrt{Q\big(\Psi({{\mathfrak S}_a(x)})\big)}}.	   		 
	   		\end{aligned}  \label{Psi}
	   	\end{equation}
		  	Let us consider $w'(x)=\frac{\sqrt{Q(x)}}{\sqrt{Q(w)}}$ with $ w(x_*)=x_*$.
				Then $w(x)=x$ is its solution.	  
		   	$\Psi\big({\mathfrak S}_a(x))$ satisfies \eqref{Psi} and  
				$\Psi\big({\mathfrak S}_a(x_*))=\Psi\big(z_*)=x_*$.  Hence 
				$\Psi\big({\mathfrak S}_a(x))=x$ and  ${\mathfrak S}_a(x)$ is injective.
	\end{proof}
		We have in the same way 
	\begin{thm}
              Let $\Omega$ be a type (2). 
		Then ${\mathfrak S}_a(x)$ is a conformal bijective mapping from $\Omega$ to 
		$\{z \in \mathbb{C}; 0<\Im z<\mu\}$.   	   \label{th5.4}
	   \end{thm}
	      $x=\Psi(z) \; ( z\in { \mathbb C} _+)$  is the inverse of $z={\mathfrak S}_a(x)\ ( x \in \Omega)$. 
    Let $\varphi(z,x)$ be a solution of \eqref{phi} and
		\begin{equation}
	   		\begin{aligned}
	   		\Sigma(x)=\{z; 0<\Im (z+ {\mathfrak S}_a(x))<\mu  \}.	   
				\end{aligned}		 
      \end{equation}
		It follows from ${\mathfrak S}_a(\varphi(z,x))={\mathfrak S}_a(x)+z$ that
		 $\varphi(z,x)=\Psi\big({\mathfrak S}_a(x))+z \big)$ 
		and $\varphi(z,x)$ is holomorphic in $	\Sigma(x)$. 
	\section{\normalsize  Appendix}
			{\it  Proof of Proposition  \ref{prop2.4}. } 
		We have
     	\begin{equation*}
     		\begin{aligned}\ &
    		 \frac{d}{dx} (Gf)(x,h)=    \sqrt{Q(x)} '\int_{-\infty}^{0} 
  		  \exp(\frac{2t}{h})\frac{f(\varphi(t,x))}{{Q(\varphi(t,x))}}dt \\  &+ 
 	   \sqrt{Q(x)} \int_{-\infty}^{0} 
    		  \exp(\frac{2t}{h})\frac{\partial}{\partial x}\Big(\frac{f(\varphi(t,x))}{{Q(\varphi(t,x))}}\Big)dt \\
	= & \frac{Q'(x)}{2Q(x)}( Gf)(x,h)+  \sqrt{Q(x)} \int_{-\infty}^{0} 
    		  \exp(\frac{2t}{h})\frac{\partial}{\partial x}\Big(\frac{f(\varphi(t,x))}{{Q(\varphi(t,x))}}\Big)dt 
   		\end{aligned}
    	\end{equation*}
	and
	\begin{equation*}
     		\begin{aligned}\ &
    		 \frac{d}{dx} (Gf)(x,h)-\frac{Q'(x)}{2Q(x)}( Gf)(x,h)= \int_{-\infty}^{0}\exp(\frac{2t}{h}) (I(t,x)+II(t, x))dt
   		\end{aligned}
    	\end{equation*}
     	\begin{equation*} \left \{
			\begin{aligned}
			I(t,x)=&	\sqrt{Q(x)}	\frac{f'(\varphi(t,x))\varphi_x (t,x) }{{Q(\varphi(t,x))}} \\
			II(t,x)=& -\sqrt{Q(x)}(\frac{Q'(\varphi(t,x))\varphi_x(t,x)}{{Q(\varphi(t,x)^2}}\Big) f(\varphi(t,x)).
			\end{aligned}\right .
		\end{equation*}	
	It follows from $ \varphi_x(t,x)=\frac{\sqrt{Q(x)}}{\sqrt{Q(\varphi(t,x))}}=\sqrt{Q(x)}\ \varphi_t(t,x)$ that
		\begin{equation*}  \left \{
			\begin{aligned} &
		I(t,x)=\frac { Q(x)f'(\varphi(t,x)\varphi_t(t,x)}{{Q(\varphi(t,x))}} 
		=\frac{Q(x)\frac{\partial}{\partial t}f(\varphi(t,x))} {{Q(\varphi(t,x))}} \\
	&	II(t,x	=-{Q(x)}(\frac{Q'(\varphi(t,x))\varphi_t(t,x)}{Q(\varphi(t,x)^{2}}\big)f(\varphi(t,x))
			\end{aligned} \right.
		\end{equation*}			
    and by integration by parts
		\begin{equation*}
			\begin{aligned} &
     		   \int_{-\infty}^{0}\exp(\frac{2t}{h})I(t,x)dt
     		  = -Q(x)\int_{-\infty}^{0}\frac{\partial}{\partial t}
     		   \big(\frac{\exp(\frac{2t}{h})}{Q(\varphi(t,x))}\big)f(\varphi(x,t))dt + f(x).
     		\end{aligned}
     	\end{equation*}
     	 We have  
  	\begin{equation*}
     		\begin{aligned}\ &
     	-Q(x)	\frac{\partial}{\partial t}\Big(\frac{\exp(\frac{2t}{h})}{Q(\varphi(t,x))}\big)
    			f(\varphi(x,t))\\	=&\exp(\frac{2t}{h})\big(\frac{-2Q(x)}{h Q(\varphi(t,x))}
    			+Q(x)\frac{Q'(\varphi(t,x))\varphi_t(t,x)}{Q(\varphi(t,x))^2}\Big)f(\varphi(x,t)),
       		\end{aligned}
     	\end{equation*}
		\begin{equation*}
			\begin{aligned}
     	\int_{-\infty}^{0}\exp(\frac{2t}{h})I(t,x)dt= &-\dfrac{2\sqrt{Q(x)}}{h}(Gf)(x,h) \\
  		  + Q(x) \int_{-\infty}^{0}\exp(\frac{2t}{h}) &\frac{Q'(\varphi(t,x))\varphi_t(t,x)}{Q(\varphi(t,x))^2}\Big)f(\varphi(t,x)dt+f(x)
			\end{aligned}
		\end{equation*}
	and 
		\begin{equation*}
			\begin{aligned}
			\int_{-\infty}^{0}\exp(\frac{2t}{h}) (I(t,x)+II(t, x))dt=-\dfrac{2\sqrt{Q(x)}}{h}(Gf)(x,h)+f(x).
			\end{aligned}
		\end{equation*}			
		Thus we get 
			$\dfrac{d}{dx} (Gf)(x,h)+\big(\dfrac{2\sqrt{Q(x)}}{h}-\dfrac{Q'(x)}{ 2Q(x)}\big)(Gf)(x,h)=f(x)$.    \qed
	
\end{document}